\documentclass[11pt,twoside]{amsart}
\usepackage{a4}
\usepackage {amssymb, latexsym, amsthm, amsmath, amsfonts, amscd}
\usepackage{graphicx}

\newtheorem{theorem}{Theorem}[section]

\newtheorem{definition}[theorem]{Definition}

\newtheorem{lemma} [theorem]{Lemma}

\newtheorem{proposition}[theorem]{Proposition}

\newcommand{\Z}{\mbox{$\mathbb Z$}}

\newcommand{\F}{\mbox{$\mathbb F$}}  
\newcommand{\N}{\mbox{$\mathbb N$}}     

\newcommand{\C}{\mbox{$\mathbb C$}}     

\newcommand{\rad}{\operatorname{rad}}
\newcommand{\Quad}{\operatorname{Quad}}
\newcommand{\sign}{\operatorname{sign}}
\newcommand{\diag}{\operatorname{diag}}
\newcommand{\Inf}{\operatorname{Inf}}
\newcommand{\Hom}{\operatorname{Hom}}
\newcommand{\GL}{\operatorname{GL}}
\newcommand{\tr}{\operatorname{tr}}

\def\rem{\refstepcounter{theorem}\paragraph{{\bf Remark \thetheorem}}}

\vspace{5cm}

\begin{document}
  
\title{Characters of real special $2$-groups}
\author{Dilpreet Kaur}
\address{Indian Institute of Science Education and Research, Knowledge City, Sector-81, Mohali 140306 INDIA.}
\thanks{The first named author gratefully acknowledges the Senior Research 
Fellowship from Council of Scientific \& Industrial Research, India
for the financial support during the period of this work.}
\email{dilpreetkaur@iisermohali.ac.in}
\author{Amit Kulshrestha}
\address{Indian Institute of Science Education and Research, Knowledge City, Sector-81, Mohali 140306 INDIA.}
\email{amitk@iisermohali.ac.in}
\date{}

\begin{abstract}  
It is well-known that special $2$-groups can be described in terms of quadratic
maps over fields of characteristic $2$.
In this article we develop methods to compute conjugacy classes, complex representations and characters
of a real special $2$-group {\it using quadratic maps alone}.
\end{abstract}
 
\maketitle

\section{Introduction}\label{Section_Introduction}

Special $2$-groups are the $2$-groups for which the commutator subgroup, the Frattini subgroup and the centre,
all three coincide and are isomorphic to an elementary abelian group. A particular case is that of {\it extraspecial $2$-groups}
where, in addition, the centre is required to be of order $2$. Non-abelian groups of order 8 and their
central products are examples of extraspecial $2$-groups \cite[\S 3.10.2]{Wilson}. 
Special $2$-groups can be described in terms of quadratic maps between vector spaces over the field of order 
$2$ \cite[\S1.3]{ObedPaper}. \\

A group $G$ is called a {\it real group} if for each $x\in G$, the conjugacy classes of $x$ and $x^{-1}$ are same. 
A {\it strongly real group} is the one in which every element can be expressed
as a product of at most two elements of order $2$. Every strongly real group is real. \\

Recently special $2$-groups have been studied to establish that there are infinitely many strongly real groups
which admit complex symplectic representations, and vice-versa, there are infinitely many
groups which are not strongly real and do not
admit symplectic representations \cite{DA}. This generates an interest in the
computation of conjugacy classes, representations and character table of special $2$-groups.
In this article we explore the description of special $2$-groups as quadratic maps 
to make these computations for real special $2$-groups. 
Our methods to compute representations, characters and conjugacy classes
can be implemented directly on the quadratic maps associated to special $2$-groups.
These methods are based on the understanding of representations of extraspecial $2$-groups. 
The key point of our proofs lies in the demonstration that the representations of extraspecial $2$-groups 
can indeed be patched together to construct all representations of real special $2$-groups. This is done by converting quadratic maps
to quadratic forms by composing them with suitable linear maps. A crucial step is a refinement in a result of Zahinda
\cite[Prop. 3.3]{ObedPaper}. \\

The article is organized as follows: In \S \ref{Section_Quadratic-maps-and-special-$2$-groups} we
define quadratic maps, special $2$-groups and recall the connection between 
them. We also
describe how one can construct quadratic forms from these quadratic maps and utilize the understanding
of extraspecial $2$-groups to construct all complex representations of real special $2$-groups, up to equivalence. 
Then in \S \ref{Section_Representations-of-real-special-$2$-groups} and \S \ref{Section_Character-table-of-real-special-$2$-groups}
we describe representations and characters of real special $2$-groups in terms of representations and characters of extraspecial $2$-groups.
A computation of conjugacy classes of special $2$-groups is done in 
\S \ref{Subsection_Conjugacy-classes-of-real-special-$2$-groups}. We conclude the article with \S \ref{example} by an explicit
computation of character table of a particular real special $2$-group following the methods developed in preceding sections. \\

The main results of this article are Theorems \ref{complete_list_representations}, \ref{character-table}, 
\ref{conjugacy_classes_not-in-rad}  and  \ref{conjugacy_classes_rad}. Theorem \ref{complete_list_representations} describes all non-linear
irreducible representations of real special $2$-groups, Theorem \ref{character-table} describes all irreducible
characters and Theorems \ref{conjugacy_classes_not-in-rad} and \ref{conjugacy_classes_rad} describe conjugacy classes of these groups.
Detailed statements of these theorems carry quite a bit of notation, and we avoid
that in the introductory section.

\section{Quadratic maps and special $2$-groups} \label{Section_Quadratic-maps-and-special-$2$-groups}
Throughout this article $\F$ denotes a field of characteristic $2$ and $\F_2$ denotes the field with two elements. Let $V$ and $W$ be vector spaces over $\F$.
A map $q : V \to W$ is called a {\it quadratic map} if $q(\alpha v) = \alpha^2 q(v)$ for all $v \in V$, for all $\alpha \in \F$
and the map $b_q : V \times V \to W$ defined by $b_q (v, w) = q(v+w)- q(v) - q(w)$ is bilinear.
The bilinear map $b_q$ is called the {\it polar map} associated to the quadratic map $q$.
We denote by $\langle b_q(V \times V) \rangle$ the $\F$-subspace of $W$ generated by the image of $b_q$.
The set of quadratic maps between $V$ and $W$ is denoted by $\Quad(V, W)$.
The subspace $\rad(b_q) := \{v \in V : b_q (v, w) = 0 ~\forall w \in V\}$ of $V$ is called the {\it radical} of $(V, q)$.
A quadratic map $q$ is said to be {\it regular} if $\rad(b_q)=0$. A quadratic map is called a {\it quadratic form} if $W = \F$. For the properties of quadratic forms over a field of characteristic $2$, we refer to \cite{HL} and \cite{Pfister}.\\

For a group $G$, let $Z(G)$ denote its centre and $[G, G]$ denote its derived subgroup.
Let $\Phi(G)$ be the Frattini subgroup of $G$.
A $2$-group $G$ is called a {\it special $2$-group} if $\Phi(G) = [G, G] = Z(G)\cong \left(\frac{\Z}{2\Z}\right)^n$ for some $n\in \N$. Note that these conditions imply that the quotient
$\frac{G}{Z(G)}$ and the centre $Z(G)$ both are elementary abelian $2$-groups \cite[Ch. 5, Th. 1.3]{GorensteinBook}.
In what follows, we describe special $2$-groups through quadratic maps. 

\subsection{Quadratic maps associated to special $2$-groups} \label{Subsection_Quadratic-maps-associated-to-special-$2$-groups}
We briefly recollect the connection between special $2$-groups and quadratic maps.
For details we refer to \cite[\S 1.3]{ObedPaper}.
Let $G$ be a special $2$-group. Let $W := Z(G)$ and $V := \frac{G}{Z(G)}$. Since $W$ and $V$
are elementary abelian $2$-groups, we may regard them as vector spaces over $\F_2$. 
Consider the map $q: \frac{G}{Z(G)}\rightarrow Z(G)$ defined by $q(xZ(G))=x^2$ for $xZ(G)\in \frac{G}{Z(G)}$.
The map $q$ is a well defined quadratic map and its {\it polar map} $b_q : V \times V \to W$ is given by 
$b_q(xZ(G),yZ(G))=xyx^{-1}y^{-1}$, where $x,y\in G$. As a quadratic map $q$ is {\it regular} and 
$\langle b_q(V \times V) \rangle = W$ \cite[Lemma 2.3]{DA}.
The map $q$ is called the {\it quadratic map associated to the special $2$-group $G$}. Conversely,
the following theorem asserts that a regular quadratic map between two finite dimensional $\F_2$-vector spaces uniquely defines
a special $2$-group provided $\langle b_q(V\times V)\rangle=W$.

\begin{theorem} [\cite{ObedPaper}, Th. 1.4] \label{special-2-group-of-a-quad-map}
Let $V$ and $W$ be two finite dimensional vector spaces over $\F_2$ and $q : V\rightarrow W$ be a regular quadratic map.
Let $b_q$ be the polar map of $q$. If $\langle b_q(V\times V)\rangle=W$ then there exists a special $2$-group $G$
such that the quadratic map associated to $G$ is $q$. Such a group is unique up to isomorphism.
\end{theorem}

Let $V$ and $W$ be finite dimensional vector spaces over $\F_2$. 
In what follows, we regard $V$ as an additive group acting trivially on $W$. In this set up we consider
the group $Z^n(V, W)$ of $n$-cocycles of $V$ with coefficients in $W$ and its subgroup 
$B^n(V, W)$ of $n$-coboundaries. The {\it $n^{\rm th}$ cohomology group} $H^n(V, W)$ for the trivial action of $V$ on $W$ is the
quotient of $Z^n(V, W)$ by $B^n(V, W)$. Our interest in this article is restricted to the case $n = 2$.
Every element in $H^2(V, W)$ can be represented by a normal $2$-cocycle. By definition,
{\it a normal $2$-cocycle} is a map $c: V\times V\to W$ satisfying following conditions:
\begin{enumerate}
 \item $c(v_2,v_3)-c(v_1+v_2,v_3)+c(v_1,v_2+v_3)-c(v_1,v_2)=0$
\item $c(v,0) = c(0,v)=0$
\end{enumerate}
where $v,v_1,v_2,v_3\in V$. The image in $H^2(V, W)$ of a normal $2$-cocycle $c \in Z^2(V, W)$ is denoted by $[c]$.


\begin{proposition}[\cite{ObedPaper}, Prop. 1.2] \label{H_Q}
Let $V$ and $W$ be two vector spaces over $\F_2$ and $c: V \times V \to W$ be a normal $2$-cocycle.
Then the map  $q_c : V \to W$ defined by $q_c(x) = c(x, x)$ is a quadratic map.
If $V$ is finite dimensional over $\F_2$ then the correspondence
$q_c {\longleftrightarrow} [c] \in H^2(V, W)$ is a bijection between $\Quad(V, W)$ and $H^2(V, W)$.
 \end{proposition}
 
\rem \label{group_operation} 
The bijection $q_c {\longleftrightarrow} [c]$ between $\Quad(V,W)$ and $H^2(V, W)$ as in Prop. \ref{H_Q} is denoted by $\phi$.
It allows us to associate a special $2$-groups to 
quadratic maps. The explicit description is as follows. Let $q : V\rightarrow W$ be a finite dimensional quadratic map satisfying $\langle b_q(V\times V)\rangle=W$.
Let $c$ be a normal $2$-cocycle on $V$ with coefficients in $W$ such that $\phi(q) = [c]$. On
the Cartesian product $V \times W$, consider the binary operation defined by
\begin{center}
$(v,w).(v^{\prime},w^{\prime})=(v+v^{\prime}, c(v,v^{\prime})+w+w^{\prime})$
\end{center}
for all $v, v^{\prime} \in V$ and $w, w^{\prime} \in W$. An easy calculation shows that the above binary operation defines a group structure on $V \times W$.
We denote this group by $G_q$. The identity element of $G_q$ is $(0,0)$ and the inverse of $(v,w)$
in $G_q$ is $(v, c(v,v)+w)$. The group $G_q$ turns out to be a special $2$-group and the quadratic map
associated to $G_q$ is $q$. It is called the {\it group associated to the quadratic map $q : V\rightarrow W$.} 
We shall drop the subscript $q$ from the notation $G_q$ whenever the description of $q$
is evident. \\

In view of Remark \ref{group_operation}, we fix the following notation for elements of special $2$-groups defined by
a quadratic map $q:V\to W$.
\begin{center}
 \begin{tabular}{|c|c|c|}
\hline 
{\bf Type of element}  & {\bf Notation} \\
\hline Arbitrary element & $(v, w) \in V \times W = G$ \\
\hline Central element & $(0, w) \in W = Z(G)$ \\
\hline Element in $\frac{G}{Z(G)}$ & $(v, 0) \in V = \frac{G}{Z(G)}$ \\
\hline
\end{tabular}
\end{center}
\vspace*{0.5cm}

\subsection{Extraspecial $2$-groups} Special $2$-groups whose
centre is of order $2$ are called {\it extraspecial $2$-groups}. They play an important role in
the understanding of representations of real special $2$-groups. The structure of extraspecial $2$-groups 
is well understood. To describe that we need to define the notion of central product.
Let $H$ and $K$ be two finite groups with isomorphic centers and $\zeta:Z(H)\to Z(K)$ be a group isomorphism.
A {\it central product} $G = H \circ_{\zeta} K$ of $H$ and $K$ with respect to the isomorphism $\zeta$ is defined to be the quotient
 $\frac{H \times K}{N}$ where $N:=\langle\{(h,k)\in Z(H)\times Z(K)~:~\zeta(h)k^{-1}=1\}\rangle$.  If the isomorphism $\zeta$ is evident in a context then we drop the subscript $\zeta$ and denote the central product by $H\circ K$. \\

\rem \label{classification-of-extraspecial-2-groups}
Since the order of center of an extraspecial $2$-group is $2$, the quadratic map $q$ associated to an extraspecial $2$-group
is a quadratic form. The quadratic form associated to the extraspecial $2$-group $D_4$,  the dihedral group of order $8$, is  given by $q_1(x,y)=xy$ and 
the one associated to $Q_2$, the quaternion group of order $8$,  is given by $q_2(x,y)=x^2+xy+y^2$. 
Following the notation of \cite{HL} we write $q_1 = [0,0]$ and $q_2 = [1,1]$. From the classification of regular quadratic forms
 it follows that for each $n\in \N$ 
there are exactly two extraspecial $2$-groups of order $2^{2n+1}$, namely 
$D_{4}\circ D_{4}\circ \cdots \circ D_{4}$ ($n$ copies of $D_4$)
and  $Q_{2}\circ D_{4}\circ \cdots\circ D_{4}$ ($n-1$ copies of $D_4$) (see \cite[\S 3.10.2]{Wilson}). 
Here $H\circ K$ denotes the central product of two groups $H$ and $K$ with isomorphic centers. \\

\subsection{From extraspecial to special}\label{Subsection_From-extraspecial-to-special}
Let $G$ be a special $2$-group and $q:V \rightarrow W$ be the quadratic map associated to $G$. 
For non-zero $s \in \Hom_{\F_2}(W,\F_2)$ we denote $s_*(q) := s\circ q:V\rightarrow \F_2$.
Indeed $s_*(q)$ is a quadratic form and its polar form is 
$b_{s_*(q)} := s\circ b_{q}:V\times V\rightarrow \F_2$. The form $s_*(q)$ is called the {\it transfer of $q$ by $s$}. 
If the radical $\rad(b_{s_*(q)})$ vanishes under $s_*(q) : V \to \F_2$ then $s_*(q)$ induces
a quadratic form $q_{s} : V_s:=\frac{V}{\rad(b_{s_*(q)})}\rightarrow \F_2$ as follows: 
$q_{s}(\epsilon_s(v)) = s_*(q)(v)$, $v\in V$; where $\epsilon_s : V \rightarrow V_s$ denotes the canonical surjection.
\begin{lemma} \label{radical-vanishing-gives-nice-q_s}
Let $G$ be a special $2$-group and $q$ be the quadratic map associated to $G$. Suppose that
the radical $\rad(b_{s_*(q)})$ vanishes under $s_*(q) : V \to \F_2$. Then
the quadratic form $q_s : V_s \to \F_2$ is regular and the polar map $b_{q_s} : V_s \times V_s \to \F_2$ is surjective.
\end{lemma}
\begin{proof}
We first show that the quadratic form $q_s : V_s \to \F_2$ is regular. For $v, w\in V$, we compute 
\begin{align*}
b_{q_s}(\epsilon_s({v}),\epsilon_s({w})) &= q_{s}(\epsilon_s({v}))+q_{s}(\epsilon_s({w}))-q_{s}
(\epsilon_s({v})+\epsilon_s({w}))\\
&= s(q({v}))+s(q({w}))-s(q({v}+{w}))\\
&= s(q({v})+q({w})-q({v}+{w}))\\
&= s(b_{q}({v},{w}))\\
 &= b_{s_*(q)}(v, w)
 \end{align*}
 Let $\epsilon_s({v})\in \rad(b_{q_s})$. Then from the above computation, we conclude that $b_{s_*(q)}(v,w_i)=0$ for some set $\{w_i\}_{i=1}^{r}$ of coset representatives of $V$ in $\rad(b_{s_*(q)})$. Let $w\in V$ be any arbitrary element. Then we write $w=w_i+w^{\prime}$ for a suitable $w^{\prime}\in \rad(b_{s_*(q)})$ and $1\leq i\leq r$. Then
 $$b_{s_*(q)}(v,w)=b_{s_*(q)}(v,w_i)+b_{s_*(q)}(v,w^{\prime})=0$$
Therefore $v \in \rad(b_{s_*(q)})$ and $\epsilon_s({v})=0$. Thus $\rad(b_{q_s})$ is the trivial subspace of $V_s$ and the quadratic form $q_s$ is regular. \\

As $G$ is special $2$-group, by \cite[Lemma 2.3]{DA}, $\langle b_{q}(V\times V)\rangle=W$. 
Since $s:W\rightarrow \F_2$ is non-zero, $b_{q_s}(V_s\times V_s) = s(b_{q}(V\times V)) = \F_2$. 
\end{proof}

With all notations as above, we consider the quadratic form $q_s : V_s \to \F_2$.
By Lemma \ref{radical-vanishing-gives-nice-q_s} and
Theorem \ref{special-2-group-of-a-quad-map} there exists
a special $2$-group $G_s$ such that the quadratic map associated to $G_s$ is $q_s$. We denote by $c_s$ a normal
$2$-cocycle such that $\phi(q_s)=[c_s]$ 
(see Prop. \ref{H_Q} and Remark \ref{group_operation}). \\

\rem \label{real_special}
If $G$ is a real special $2$-group then $s_*(q)(\rad(b_{s_*(q_G)})) = 0$. We refer to \cite[\S 3]{DA}
for details. Recall that a group $G$ is called {\it real} if for each $x\in G$, the conjugacy classes of $x$ and $x^{-1}$ are same. \\

We now record results on real special $2$-groups that will be useful later in the article. The following result characterizes real special $2$-groups in terms of the associated quadratic map.
 
\begin{theorem}[\cite{ObedPaper}, Th. 2.1]\label{real-group-criteria}
Let $G$ be a special $2$-group and $q : V \to W$ be the quadratic map associated to $G$. 
The following assertions are equivalent:
\begin{enumerate}
\item[$i.$] The group $G$ is real. 
\item[$ii.$] For all $v \in V$, there exists $v^{\prime} \in V$ such that $q(v^{\prime})=q(v+v^{\prime})$.
\end{enumerate}
\end{theorem}

The following result for real special $2$-groups is implicit in the proof of \cite[Prop. 3.3]{ObedPaper} 

\begin{lemma}\label{G_s-extraspecial} 
Let $G$ be a real special $2$-group and $q : V \to W$ be the quadratic map associated to $G$. 
For $0\neq s \in \Hom_{\F_2}(W,\F_2)$ let $q_s : V_s \to \F_2$ be the quadratic form as in 
Lemma \ref{radical-vanishing-gives-nice-q_s}. Let $G_s$ denote the special $2$-group associated to $q_s$. Then

\begin{enumerate}
\item[$i.$] The group $G_s$ is extraspecial $2$-group.
\item[$ii.$] $V_s \simeq \frac{G_s}{Z(G_s)}$.
\end{enumerate}
\end{lemma}

\begin{proof}
 \begin{enumerate}
\item[$i.$]

Since the quadratic form $q_s$ associated to the special $2$-group $G_s$ takes values in $\F_2$, it follows that $|Z(G_s)| = 2$.
Therefore $G_s$ is extraspecial.
\item[$ii$.] We recall from Remark \ref{group_operation} that the underlying set of $G_s$ is the Cartesian product $V_s\times \F_2$ and its group operation is given by $(v,w)(v^{\prime},w^{\prime})=(v+v^{\prime}, c_s(v,v^{\prime})+w+w^{\prime})$, where $v,v^{\prime}\in V_s; ~w,w^{\prime}\in \F_2$. Here $c_s$ is a normal $2$-cocycle whose cohomology class corresponds 
to the quadratic form $q_s$ (see Prop. \ref{H_Q}). Regarding $V_s$ as an abelian group under addition, 
we consider the group
homomorphism $\xi: G_s\to V_s$ given by $\xi(v,w)=v$. Clearly $\xi$ is a surjection. Therefore $\frac{G_s}{\ker\xi}\cong V_s$. Now to prove the result, we need to show that $\ker \xi=Z(G_s)$. If $(v,w)\in \ker\xi$ then $v=0$. Thus for all $(v^{\prime},w^{\prime})\in G_s$, we have
\begin{center}
$$(0,w)(v^{\prime},w^{\prime})=(0+v^{\prime}, c_s(0,v^{\prime})+w+w^{\prime})=(v^{\prime}+0, c_s(v^{\prime},0)+w+w^{\prime})=(v^{\prime},w^{\prime})(0,w)$$
\end{center}
This confirms that $Z(G_s)$ contains $\ker\xi$. Now for the reverse inclusion, let $(v,w)\in Z(G_s)$. For all $(v^{\prime},w^{\prime})\in G_s$, we have
\begin{align*}
 (v,w)(v^{\prime},w^{\prime})&=(v^{\prime},w^{\prime})(v,w)\\
 \Rightarrow (v+v^{\prime}, c_s(v,v^{\prime})+w+w^{\prime})&=(v^{\prime}+v, c_s(v^{\prime},v)+w^{\prime}+w)\\
 \Rightarrow c_s(v,v^{\prime})&=c_s(v^{\prime},v).
\end{align*}
By \cite[Prop. 1.4]{ObedPaper}, we have $b_{q_s}(v,v^{\prime})=c_s(v,v^{\prime})-c_s(v^{\prime},v)$. From the above calculation, we have $b_{q_s}(v,v^{\prime})=0$ for all $v^{\prime}\in V_s$. Thus $v\in \rad(b_{q_s})$. Since $q_s$ is regular quadratic form, $v=0$ and therefore $(v,w)\in \ker\xi$.
\end{enumerate}
\end{proof}

\section{Representations of real special $2$-groups}\label{Section_Representations-of-real-special-$2$-groups}
In this section we describe irreducible representations of real special $2$-groups. We begin with 
representations of degree one.
\subsection{Linear representations of special $2$-groups}
Throughout this article, by a {\it linear representation} we mean a representation of degree one.
With this nomenclature, representations of degree at least two will be called {\it non-linear representations}. 
Finding linear representations of special $2$-groups is elementary and is based on the following well-known results.

\begin{theorem}[\cite{JL}, Th. 17.11]\label{linear-representations-through-abelianization}
Let $G$ be a finite group.
The linear representations of $G$ are precisely the lifts to $G$ of the irreducible representations of $\frac{G}{[G,G]}$. 
In particular, the number of distinct linear representations of $G$ equals the index of $[G,G]$ in $G$. 
\end{theorem}

Let $A=[a_{ij}]_{n\times n}$ and $B$ be two matrices over a field. Then the tensor product of matrices $A$ and $B$ is
defined by
$$A\otimes B=\begin{pmatrix}
              a_{11}B & a_{12}B & \cdots & a_{1n}B\\
              a_{21}B & a_{22}B & \cdots & a_{2n}B\\
              \vdots & \vdots & \ddots & \vdots\\
              a_{n1}B & a_{n2}B & \cdots & a_{nn}B\\
             \end{pmatrix}
$$

The following theorem provides a description of representations of direct products of groups in terms of irreducible
representations of direct factors.
\begin{theorem}[\cite{GorensteinBook}, Ch. 3, Th. 7.1] \label{representations_of_direct_product}
 Let $H$ and $K$ be two finite groups and $G=H\times K$ be the direct product of $H$ and $K$. 
Let $\rho:H\rightarrow \GL(n,\C)$ and $\sigma:K\rightarrow \GL(m,\C)$ be irreducible representations of $H$ and $K$, respectively. 
Then $\rho \otimes \sigma:G\rightarrow \GL(nm,\C)$ defined by $(\rho \otimes \sigma)(h,k)=\rho(h) \otimes \sigma(k)$ for $(h,k) \in H \times K = G$ 
is an irreducible representation of $G$. 
Moreover, every irreducible representation of $G = H\times K$ is equivalent to a representation of the form $\rho \otimes \sigma$ for a suitable choice of 
$\rho$ and $\sigma$. 
\end{theorem}


\rem \label{linear-representations-of-special-2-group}
For a special $2$-group $G$ the quotient $\frac{G}{Z(G)}$ is isomorphic to a direct 
product of copies of $\frac{\Z}{2\Z}$.
Clearly $\frac{\Z}{2\Z}$ has only one non-trivial irreducible representation. Thus from Theorem \ref{representations_of_direct_product}
one can write all irreducible
representations of $\frac{G}{Z(G)}$. Now using Theorem \ref{linear-representations-through-abelianization} and the 
equality $Z(G)=[G,G]$ for special $2$-groups, one can write all linear representations of special $2$-groups. \\

Interesting part of the discussion is to describe non-linear irreducible representations of special $2$-groups.
In this article we limit ourselves to real special $2$-groups. For a real special $2$-group $G$
we shall describe non-linear irreducible representations of $G$ in terms of 
non-linear irreducible representations
of extraspecial $2$-groups $G_s$. The section \ref{Subsection_Representations-of-extraspecial-$2$-groups} concerns the representations of extraspecial $2$-groups and the section \ref{Subsection_Representations-of-real-special-$2$-groups} concerns the non-linear representations of real special $2$-groups.

\subsection{Non-linear representations of extraspecial $2$-groups}\label{Subsection_Representations-of-extraspecial-$2$-groups}
Let $H$ and $K$ be two finite groups with isomorphic centers. Let $\zeta:Z(H)\to Z(K)$ be a group isomorphism. Let $N:=\langle\{(h,k)\in Z(H)\times Z(K)~:~\zeta(h)k^{-1}=1\}\rangle$ and $G = H \circ K\cong\frac{H \times K}{N}$ be the central product of groups $H$ and $K$.
Let $\varphi: H \times K \to \GL(n, \C)$ be a representation of $H \times K$ such that
$N \subseteq \ker(\varphi)$. Then $\varphi$ can be treated as a representation $\widehat{\varphi}$ of $G$ simply by defining
$\widehat{\varphi}(\overline{(h,k)})=\varphi(h,k)$, where $\overline{(h,k)}=(h,k)N$ for all $(h,k)\in H\times K$. 
The representation $\varphi$ is irreducible if and only if $\widehat{\varphi}$ is irreducible \cite[Ch. 3, Th. 7.2]{GorensteinBook}.
Since extraspecial $2$-groups are
central products of copies of non-abelian groups of order $8$, we repeatedly use the constructions like $\widehat{\varphi}$ to describe all non-linear irreducible representations of extraspecial $2$-groups.

\rem \label{representations_of_small_extraspecial}
\begin{enumerate}
\item For the dihedral group $D_4=\langle a,b~ : ~a^4=b^2=1,bab^{-1}=a^{-1}\rangle$
the homomorphism $\rho:D_4\rightarrow \GL(2,\C)$ defined by
$$\rho(a)= \left(\begin{array}{cc} 0 & 1 \\ -1 & 0 \end{array}\right), 
~~\rho(b)= \left(\begin{array}{cc} 1 & 0 \\ 0 & -1\end{array}\right)$$
is the only non-linear irreducible representation.

\item For the quaternion group $Q_2=\langle c,d~ : ~c^4=1,d^2=c^2,dcd^{-1}=c^{-1}\rangle$  
the homomorphism $\sigma:Q_2\rightarrow \GL(2, \C)$, where
$$\sigma(c)= \left(\begin{array}{cc} 0 & 1 \\ -1 & 0 \end{array}\right), 
~~\sigma(d)= \left(\begin{array}{cc} i & 0 \\ 0 & -i \end{array}\right)$$
is the only non-linear irreducible representation. \\
\end{enumerate}

Based on this remark, we have the following :

\begin{proposition}\label{non-linear_representation_extraspecial}
 Every extraspecial $2$-group has a unique non-linear representation.
\end{proposition}
\begin{proof}
Let $\gamma_{i}: 1\leq i\leq 4$ be linear representations of the group $D_{4}$.
From Theorem \ref{representations_of_direct_product} the non-linear irreducible representation of the group 
$D_{4}\times D_{4}$ are $\rho \otimes \gamma_{i}: 1\leq i\leq 4$ and $\rho \otimes \rho$, where $\rho$ is as in Remark \ref{representations_of_small_extraspecial}(1).
The  normal subgroup $N:=\{(1,1),(a^2,a^2)\}$ of the group $D_{4}\times D_{4}$ is contained in the kernel of $\rho \otimes \rho$ 
but $N$ is not contained in the kernel of $\rho \otimes \gamma_{i}$ for all $1\leq i\leq 4$. As $D_{4}\circ D_{4}=\frac{D_{4}\times D_{4}}{N}$,
from Theorem \cite[Ch. 3, Th. 7.2]{GorensteinBook} it is evident that  $\widehat{\rho \otimes \rho}$ is the only non-linear representation of $D_{4}\circ D_{4}$, which is induced from the representation ${\rho{\otimes} \rho}$ of $D_{4}\times D_{4}$.
This generalizes to the fact 
the representation $\widehat{\rho \otimes \rho \otimes \cdots \otimes \rho}$ ($l$ copies of $\rho$) is the only non-linear representation of $D_4\circ D_4\circ \cdots \circ D_4$ ($l$ copies of $D_4$). 
Similarly the group  $Q_{8}\circ D_{4}\circ \cdots\circ D_{4}$ ($l-1$ copies of $D_4$) has unique non-linear irreducible representation, namely $\widehat{\sigma{\otimes} \rho {\otimes} \cdots {\otimes} \rho}$ ($l-1$ copies of $\rho$).
\end{proof}
\rem \label{unique_non_linear_rep}
The degree of unique non-linear representation of extraspecial $2$-group of order $2^{2l+1}$ is $2^l$.\\


The following theorem is well-known standard result.

\begin{theorem}[\cite{JL}, Th. 11.12]\label{sum_of_squares_of_reps}
Let $G$ be a finite group and $\chi_1, \chi_2,\cdots \chi_k$ be the complete list of distinct irreducible characters of $G$. Then
$\displaystyle \sum_{i=1}^{k} \chi_i(1)^{2} = |G|$.
\end{theorem}
 
\begin{lemma}\label{character_of_extraspecial_2-group}
Let $G$ be an extraspecial $2$-group of order $2^{2l+1}$ and $\chi$ be the character of its unique non-linear irreducible representation. Then 
$$\chi(g) = 
\begin{cases}
2^{l}\quad \text{if $g$ is the identity element of $G$}\\
-2^{l}\quad \text{if $g$ is a non-trivial element of $Z(G)$}\\
0 \quad \quad \text{otherwise}.
\end{cases}
$$

\end{lemma}
\begin{proof}
By Remark \ref{classification-of-extraspecial-2-groups}, the group $G$ is isomorphic to either  
$D_4\circ D_4\circ\cdots \circ D_4$ ($l$ copies of $D_4$) or $Q_2\circ D_4\circ D_4\circ\cdots \circ D_4$ ($l-1$ copies of $D_4$). 
Let $\rho$ and $\sigma$ be unique non-linear representations of $D_4$ and $Q_2$, respectively, as in Remark \ref{representations_of_small_extraspecial}.
Let $\varphi: G\to \GL(2^l,\C)$ be the non-linear irreducible representation of $G$. Then by Prop. \ref{non-linear_representation_extraspecial},  $\varphi=\widehat{\rho{\otimes}\rho{\otimes} \cdots {\otimes}\rho}$ ($l$ copies of $\rho$) if $G\cong  D_4\circ D_4\circ\cdots \circ D_4$ ($l$ copies of $D_4$) or $\varphi=\widehat{\sigma{\otimes}\rho{\otimes}\rho{\otimes} \cdots {\otimes}\rho}$ ($l-1$ copies of $\rho$) if $G\cong  Q_2\circ D_4\circ D_4\circ\cdots \circ D_4$ ($l-1$ copies of $D_4$).
 If $G\cong  D_4\circ D_4\circ\cdots \circ D_4$ ($l$ copies of $D_4$), using Theorem \ref{representations_of_direct_product}, we know that $\varphi(\bar{1})=\rho(1)\otimes\rho(1)\otimes\cdots\otimes\rho(1)$($l$ times). Therefore $\chi_{\varphi}(\bar{1})=\tr(\rho(1)\otimes\rho(1)\otimes\cdots\otimes\rho(1)$, where $\chi_{\varphi}$ is character associated to representation $\varphi$. Using Remark \ref{representations_of_small_extraspecial} and the fact that for two matrices $A$ and $B$, $\tr(A\otimes B)=\tr(A)\tr(B)$, we get $\chi_{\varphi}(\bar{1})=2^l$.
 If $\bar{1}\neq g\in Z(G)$ then $g=\overline{(a^2,1,\cdots,1)}$, where $a^2$ is the non-trivial element of $Z(D_4)$. Then we have 
 \begin{align*}
  \chi_{\varphi}(g)&=\tr(\rho(a^2)\otimes\rho(1)\otimes\cdots\otimes\rho(1))\\
  &=\tr(\rho(a^2))\tr(\rho(1))\cdots\tr(\rho(1))\\
  &=-2^l
 \end{align*}
If $g=\overline{(g_1,g_2,\cdots,g_l)}\in G \backslash Z(G)$, then for some $1\leq i\leq l$, 
we have $g_i\in D_4 \backslash Z(D_4)$ and $\rho(g_i)=0$. Thus $\chi_{\varphi}(g)=0$. This proves the result if $G\cong  D_4\circ D_4\circ\cdots \circ D_4$ ($l$ copies of $D_4$). One can prove the result for $G\cong  Q_2\circ D_4\circ D_4\circ\cdots \circ D_4$ ($(l-1)$ copies of $D_4$) on similar lines.
\end{proof}

\subsection{Non-linear representations of real special $2$-groups}\label{Subsection_Representations-of-real-special-$2$-groups} 
In this section we describe all non-linear irreducible representations of real special $2$-groups. We begin with recalling the following.
\begin{proposition}[\cite{ObedPaper}, Prop. 3.3]\label{quad-vs-rep}
 Let $G$ be a real special $2$-group and $q: V := \frac{G}{Z(G)}\rightarrow Z(G) =: W$ be the quadratic map associated to $G$. Then
\begin{enumerate}
 \item For every non-zero $s\in \Hom_{\F_2}(W, \F_2)$, there exists a non-linear irreducible representation $\varphi$ of $G$ 
such that $\varphi(G) = G_s$.
\item Conversely, for all non-linear irreducible representations $\varphi$ of $G$, 
there exists a non-zero $s\in \Hom_{\F_2}(W, \F_2)$ such that $\varphi(G)=G_s$.
\end{enumerate}
\end{proposition}

We refine the first part of Prop. \ref{quad-vs-rep} by observing that for every non-zero $s\in \Hom_{\F_2}(W, \F_2)$, 
there are exactly $|\rad(b_{s_*(q)})|$ many of inequivalent non-linear irreducible representations of $G$ 
such that $\varphi(G)=G_s$. Here $|\rad(b_{s_*(q)})|$ denotes the size of the radical $\rad(b_{s_*(q)})$.\\

Before stating next proposition, we record some definitions, which will be used later.
 \begin{enumerate}
  \item Let $c: V\times V\to W$ be a normal $2$-cocycle and $s\in \Hom_{\F_2}(W,\F_2)$. Then $s_*(c): V\times V\to \F_2$ defined by $s_*(c)(v,v^{\prime})=s(c(v,v^{\prime}))$ for all $v,v^{\prime}\in V$ is a normal $2$-cocycle. It is called the {\it transfer of $c$ by $s$}.
 \item Recall from \S \ref{Subsection_From-extraspecial-to-special} that $V_s=\frac{V}{\rad(b_{s_*(q)})}$. Let $\epsilon_s:V\to V_s$ be the canonical surjection and $c_s: V_s\times V_s \to W$ be a normal $2$-cocycle. Then $\Inf(c_s): V\times V \to \F_2$ defined by $\Inf(c_s)(v, v^{\prime})=c_s(\epsilon_s(v), \epsilon_s(v^\prime))$ for $v, v^{\prime}\in V$ is a normal $2$-cocycle. It is called the {\it inflation} of $c_s$.
 \end{enumerate}

\begin{proposition}\label{number of surjective homomorphisms}
Let $G$ be a real special $2$-group and  $q:V:=\frac{G}{Z(G)}\rightarrow Z(G) =: W$ be the quadratic
map associated to $G$. Then for every non-zero $s\in \Hom_{\F_2}(W, \F_2)$ there exists at least $|\rad(b_{s_*(q)})|$ many surjective homomorphisms form $G$ to the extraspecial $2$-group $G_s$.
\end{proposition}
\begin{proof}
Let $s \in \Hom_{\F_2}(W, \F_2)$ be a non-zero map.
Since $\rad(b_{s_*(q)})$ is a subspace of $V$, we have $|\rad(b_{s_*(q)})|=2^{k}$ for some $k\in \N$. Since the order of  $\Hom_{\F_2}(\rad(b_{s_*(q)}), \F_2)$ is same as that of $\rad(b_{s_*(q)})$, we have $2^{k}$ linear maps from 
$\rad(b_{s_*(q)})$ to $\F_2$. We enumerate these linear maps as $\{t_i\}$; $1\leq i\leq 2^k$. For rest of the proof we fix a vector space complement $V^{\prime}$ of $\rad(b_{s_*(q)})$ in $V$. Thus we write $V=\rad(b_{s_*(q)})\oplus V^{\prime}$. Define $h_i:V \rightarrow \F_2$ by $h_i(v)=t_i(x),$ where $v = (x,y) \in V$ with $x\in\rad(b_{s_*(q)})$ and $y\in V^{\prime}$. \\ 

Recall that there exists a normal $2$-cocycle $c$
such that $\phi(q) = [c]$, where $\phi$ is the isomorphism between $H^2(V,W)$ and $\Quad(V,W)$ as  
in Prop. \ref{H_Q}. Let $s\in \Hom_{\F_2}(W, \F_2)$ and $s_*(c):V\times V\rightarrow \F_2$ 
 be the transfer of $c$ by $s$. 
Let $c_{s}$ be a normal $2$-cocycle such that $\phi(q_s) = [c_{s}]$, 
where $\phi$ is the isomorphism between $H^2(V_s,\F_2)$ and $\Quad(V_s,\F_2)$ as defined in Prop. \ref{H_Q}. 
Now we have 
$$ \Inf(c_{s})(v,v)=q_{s}(\epsilon_s(v))=s(q(v))=s_*(c)(v,v)$$

Thus both $[\Inf(c_{s})]$ and $[s_*(c)]$ are preimages of the same quadratic map under the isomorphism $\phi: H^2(V,\F_2)\rightarrow \Quad(V,\F_2)$ 
given in Prop. \ref{H_Q}. Therefore $\Inf(c_{s})$ and $s_*(c)$ are cohomologous. Thus there exists a map $\lambda: V\rightarrow \F_2$ such that $\lambda(0)=0$ and 
\begin{equation}\label{lambda}
 \Inf(c_{s})(v,v^{\prime})= s_*(c)(v,v^{\prime})-\lambda(v+v^{\prime})+\lambda(v)+\lambda(v^{\prime}).
\end{equation}

We are now ready to define surjective homomorphisms from $G$ to $G_s$ for each $i; 1 \leq i \leq 2^k$. 
Consider $f_{s,i}:G\rightarrow G_s$ defined by 
$$f_{s,i}(v,w)=(\epsilon_s(v), s(w)-\lambda(v)-h_i(v))$$ for $v\in V, w\in W$.  Here for group elements of $G$ we follow the notation
as in the table of \S \ref{Subsection_Quadratic-maps-associated-to-special-$2$-groups}.
That each $f_{s,i}:G\rightarrow G_s$ is a homomorphism follows from the following
direct computation.
\begin{align*}
 f_{s,i}((v,w)(v^{\prime},w^{\prime}))&=f_{s,i}(v+v^{\prime},c(v,v^{\prime})+w+w^{\prime})\\
                                      &=(\epsilon_s(v+v^{\prime}),s(c(v,v^{\prime})+w+w^{\prime})-\lambda(v+v^{\prime})-h_i(v+v^{\prime}))\\
                                      &= (\epsilon_s(v)+\epsilon_s(v^{\prime}),c_{s}(\epsilon_s(v), 
\epsilon_s(v^{\prime})+s(w)+s(w^{\prime})-\lambda(v)-\lambda(v^{\prime})-h_i(v)-h_i(v^{\prime}))\\
                                      &=(\epsilon_s(v),s(w)-\lambda(v)-h_i(v))(\epsilon_s(v^{\prime}),
s(w^{\prime})-\lambda(v^{\prime})-h_i(v^{\prime}))\\
                                      &=f_{s,i}((v,w))f_{s,i}((v^{\prime},w^{\prime}))
\end{align*}
for $(v,w),(v^{\prime},w^{\prime})\in G$.

To check the surjectivity, let $(v_s,w_s)\in G_s$ where $v_s\in V_s$ and $w_s\in \F_2$. 
Since maps $\epsilon_s$ and $s$ are surjective, there exist $v\in V$ and $w\in W$ such that $\epsilon_s(v)=v_s$ and $s(w)=w_s$. Further, since $\lambda(v), h_i(v)\in \F_2$, there exist $w_1, w_2\in W$ such that $\lambda(v)=s(w_1)$ and $h_i(v)=s(w_2)$. The following calculation confirms that $f_{s, i}(v, w+w_1+w_2) = (v_s, w_s)$. 
\begin{align*}
 f_{s,i}(v,w+w_1+w_2)&=(\epsilon_s(v),s(w+w_1+w_2)-\lambda(v)-h_i(v))\\
                     &=(\epsilon_s(v),s(w)+s(w_1)+s(w_2)-\lambda(v)-h_i(v))\\
                     &=(v_s,w_s)
\end{align*}
Since for $i\neq j$ there exists $v\in V$ such that $h_i(v)\neq h_j(v)$, we have $f_{s,i}\neq f_{s,j}$.
Thus $f_{s, i}; 1 \leq i \leq 2^k$ are distinct homomorphisms. These are $|(\rad(b_{s_*(q)}))| = 2^k$ in number.
\end{proof}

\begin{proposition}\label{rep-spl}
Let $G$ be a real special $2$-group and  $q:V:=\frac{G}{Z(G)}\rightarrow Z(G) =: W$ be the quadratic
map associated to $G$. Then for every non-zero $s\in \Hom_{\F_2}(W, \F_2)$ 
there exist at least $|(\rad(b_{s_*(q)}))|$ many inequivalent non-linear irreducible representations $\varphi$ of $G$
such that $\varphi(G)=G_s$.
\end{proposition}
\begin{proof}
Let $s\in \Hom_{\F_2}(W, \F_2)$ be a non-zero linear map and $k$ be such that $|\rad(b_{s_*(q)})|=2^k$. 
Let $\varphi_s$ denote the unique non-linear irreducible  faithful representation of the extraspecial $2$-group $G_s$.
Then $\varphi_{s,i} := \varphi_s \circ f_{s,i}$ for $1\leq i \leq 2^k$, where $f_{s,i}$ is as in the proof of Prop. \ref{number of surjective homomorphisms}, are irreducible representations of $G$ of the degree same as that of $\varphi_s$.
Since $f_{s,i}:G\to G_s$ is surjective and $\varphi_s$ is faithful, we have $\varphi_{s,i}(G)=\varphi_s(f_{s,i}(G))=\varphi_s(G_s)\cong G_s$.
We claim that $\varphi_{s,i}$ and $\varphi_{s,j}$ are equivalent if and only if $i=j$. 
Let $\chi_{s,i}$ and $\chi_{s,j}$ be the characters of representations $\varphi_{s,i}$ and $\varphi_{s,j}$, respectively. Then $\chi_{s,i}=\chi_{s}f_{s,i}$ and $\chi_{s,j}=\chi_{s}f_{s,j}$.
Recall from \cite[p. 16, Cor. 2]{SerreRepnBook} that $\varphi_{s,i}$ and $\varphi_{s,j}$ are equivalent if and only if 
$\chi_{s,i}(v,w)=\chi_{s,j}(v,w)$ for every $(v,w)\in G$. We use it to complete the proof. \\

If $|(\rad(b_{s_*(q)}))|=1$ then there is nothing to prove. If $|(\rad(b_{s_*(q)}))|>1$, then for $i\neq j$ there exists $(v,0)\in G$, where $v\in \rad(b_{s_*(q)})$, such that $h_i(v)\neq h_j(v)$. Thus
$$f_{s,i}(v,0)=(0,-\lambda(v)-h_i(v)),~~~~~ f_{s,j}(v,0)=(0,-\lambda(v)-h_j(v))$$
Here $h_i : V \to \F_2$ is the homomorphism as defined in the proof of Prop. \ref{number of surjective homomorphisms}. \\

It follows that one of $f_{s,i}(v,0)$ and $f_{s,j}(v,0)$ is the identity element of the group $G_s$,
while the other is the non-trivial element of $Z(G_s)$. Without loss of generality we assume that 
$$
f_{s,i}(v,0)=(0,0) \in Z(G_s),~~~~f_{s,j}(v,0)=(0,1) \in Z(G_s).$$
Let the order of $G_s$ be $2^{2l+1}$. Then by Lemma \ref{character_of_extraspecial_2-group}
\begin{center}
$ \chi_{s,i}(v,0)=\chi_{s}f_{s,i}(v,0)=\chi_s(0,0)=2^{l}$\\
$\chi_{s,j}(v,0)=\chi_{s}f_{s,j}(v,0)=\chi_s(0,1)=-2^{l}$\\
\end{center}
This proves that $\varphi_{s,i}$ and $\varphi_{s,j}$ are inequivalent if $i\neq j$.
\end{proof}

The following theorem describes all non-linear irreducible representations of real special $2$-groups.

\begin{theorem}\label{complete_list_representations}
Let $G$ be a real special $2$-group and  $q:V:=\frac{G}{Z(G)}\rightarrow Z(G) =: W$ be the quadratic
map associated to $G$. Then $\{\varphi_{s, i} : s  \in \Hom_{\F_2}(W, \F_2), 1 \leq i \leq 2^k\}$ is the complete
list of non-linear irreducible representations of $G$; where $\varphi_{s, i}$ are as in 
Prop. \ref{rep-spl} and $2^k$ is the size of the radical $\rad(b_{s_*(q)})$.
\end{theorem}

\begin{proof}
Let $|G|=2^n$ and $|Z(G)|=2^m$. The number of non-zero linear maps from $Z(G)$ to $\F_2$ is $|Z(G)|-1=2^m-1$. We denote these linear maps by
$s_1, s_2, \cdots, s_{2^{m}-1}$.
First we prove that $\varphi_{s_{p},i}\nsim\varphi_{s_{q},j}$ if either $p\neq q$ or $i\neq j$.
If $p=q$ then it follows from the proof of Prop. \ref{rep-spl} that $\varphi_{s_{p},i}\nsim\varphi_{s_{q},j}$ if $i\neq j$. Suppose $p\neq q$. Recall that $\varphi_{s_{p},i}(G)\cong G_{s_p}$ and $\varphi_{s_{q},j}(G)\cong G_{s_q}$. If $G_{s_p}\ncong G_{s_q}$ then $\varphi_{s_{p},i}\nsim\varphi_{s_{q},j}$. Suppose that $G_{s_p}\cong G_{s_q}$ and $|G_{s_p}|=|G_{s_q}|=2^{2l+1}$. Since $s_p\neq s_q$, there exist $w\in W$ such that $s_p(w)\neq s_q(w)$. Then $f_{s_p,i}(0,w)=(0,s_p(w))$ and $f_{s_q,j}(0,w)=(0,s_q(w))$. Now by the same argument as in the proof of Prop. \ref{rep-spl}, without loss of generality, we assume that $\chi_{s_p,i}(0,w)=2^l$ and $\chi_{s_q,j}(0,w)=-2^l$. Here $\chi_{s_p,i}$ and $\chi_{s_q,j}$ are characters associated with representations $\varphi_{s_{p},i}$ and $\varphi_{s_{q},j}$ respectively.

To prove the theorem we use Theorem \ref{sum_of_squares_of_reps} and show that squares of degrees of representations $\varphi_{s, i}$ and linear
representations add up to $2^n$. \\

 Let $|G_{s}| = 2^{2l+1}$ 
for some non-zero linear map $s\in \Hom_{\F_2}(W, \F_2)$. Since $G_{s}$ are extraspecial $2$-groups, we have $|Z(G_{s})|=2$.
Using Lemma \ref{G_s-extraspecial}(2), we have
$$ \frac{|\frac{G}{Z(G)}|}{|\frac{G_{s}}{Z(G_{s})}|}=|\rad b_{s_*(q)}|$$
Therefore $|\rad b_{s_*(q)}|=\frac{2^{n-m}}{2^{2l}}$.
By Theorem \ref{rep-spl}, there are at least $|\rad b_{s_*(q)}|$ 
number of non-linear irreducible inequivalent representations $\varphi$ of $G$ such that $\varphi(G)\cong G_{s}$. 
The degree of these representations is same as the degree of the faithful representation $\varphi_{s}$ of $G_{s}$. 
We know that the degree of the non-linear faithful representation of extraspecial $2$-groups of 
order $2^{2l+1}$ is $2^{l}$ (see Remark \ref{unique_non_linear_rep}). 
Now we compute the sum of squares of degrees of irreducible representations of $G$.   
\begin{align*}
\left|\frac{G}{Z(G)}\right|.1^2+\sum_{j=1}^{2^{m}-1}|\rad b_{s_{j_*}(q)}|.(2^{l_j})^2 &= \left|\frac{G}{Z(G)}\right|.1^2+\sum_{j=1}^{2^{m}-1}\frac{|\frac{G}{Z(G)}|}
{|\frac{G_{s_j}}{Z(G_{s_j})}|}.(2^{l_j})^2\\
                                                                         &= 2^{n-m}+\sum_{j=1}^{2^{m}-1}\frac{2^{n-m}}{2^{2l_j}}.(2^{l_j})^2\\
                                                                         &= 2^{n-m}+\sum_{j=1}^{2^{m}-1} 2^{n-m}\\
                                                                         &= 2^{n-m}+ (2^{m}-1)2^{n-m}\\
                                                                         &= 2^{n}\\
                                                                         &= |G|.
 \end{align*}
 Therefore $G$ can not afford  non-linear representations except $\varphi_{s, i}$ and $\{\varphi_{s, i} : s  \in \Hom_{\F_2}(W, \F_2), 1 \leq i \leq 2^k\}$ is the complete list of non-linear irreducible representations of $G$.
\end{proof}

\section{Character table of real special $2$-groups}\label{Section_Character-table-of-real-special-$2$-groups}
The aim of this section is to provide a method to write the character table of real special $2$-groups using quadratic
maps alone.

\subsection{Characters of real special $2$-groups}
In what follows, we keep the notations of Prop. \ref{rep-spl}.
\begin{proposition} \label{spl_to_extraspl}
Let $G$ be a real special $2$-group.
Let $\chi_{s,i}$ be the character of the representation 
$\varphi_{s,i}$, as described in Prop. \ref{complete_list_representations}
Let the order of $G_s$ be $2^{2l+1}$. Then
 $$ \chi_{s,i}(g)=\left\{ \begin{array}{lll}
                    2^l  & \text{if $f_{s,i}(g)=1$}\\
                           -2^l &  \text{if $f_{s,i}(g)$ is the non-trivial element of $Z(G_s)$}\\ 
                          0 & \text{otherwise}
                   \end{array} \right. $$

\end{proposition}

\begin{proof}
Let $\varphi_s$ be the non-linear irreducible representation of the extraspecial $2$-group $G_s$ and $\chi_s$ be the character of  $\varphi_s$.
Since $\varphi_{s,i}=\varphi_s \circ f_{s,i}$, we have $\chi_{s,i}=\chi_s \circ f_{s,i}$.
Now the proof follows immediately from Lemma \ref{character_of_extraspecial_2-group}.
\hfill $\square$
\end{proof}
\rem \label{representations}
Let $\diag(1,1,\cdots,1)$ denote the identity matrix of order $2^l$ and $\diag(-1,-1,\cdots,-1)$ denote the diagonal matrix of order $2^l$ with diagonal entries equal to $-1$. Then
$$ \varphi_{s,i}(g)=\left\{ \begin{array}{ll}
                              \diag(1,1,\cdots,1)  & \text{if $\chi_{s,i}(g)=2^l$}\\
                              \diag(-1,-1,\cdots,-1)       &  \text{if $\chi_{s,i}(g)=-2^l$}
                                   \end{array} \right. $$
 
{\definition \label{sign} For the character $\chi_{s,i}$, we define 
$$ \sign(\chi_{s,i}(g)) =\left\{\begin{array}{ll} -1 & \text{if $\chi_{s,i}(g)=-2^l$}\\
                                   1 & \text{if $\chi_{s,i}(g)=2^l$}
                 \end{array} \right. $$}

\begin{theorem} \label{character-table}
Let $G$ be a real special $2$-group and $q: V \to W$ be the quadratic map associated to $G$. 
Let $s \in \Hom_{\F_2}(W, \F_2)$. Then
  \begin{enumerate}
\item  $\chi_{s,i}(v,w)\neq 0$ if and only if $v\in \rad(b_{s_*(q)})$.
\item For $(0,w)\in Z(G)$ we have
$$ \chi_{s,i}(0,w)=\left\{ \begin{array}{lll}
                    2^l  & \text{if $s(w)=0$}\\
                           -2^l &  \text{if $s(w)=1$}
                   \end{array} \right. $$
where $l$ is defined by $|G_s| = 2^{2l+1}$. 
                                      
\item Let $\{v_1,v_2,\cdots, v_k\}$ be a basis of $\rad(b_{s_*(q)})$. 
Then 
$$ \chi_{s,i}(v_j,0)=\left\{ \begin{array}{ll}
                            -2^{l} & \text{if $A_{j,i}=1$}\\
                             2^{l} & \text{if $A_{j,i}=0$} 
                           \end{array} \right. $$
 where $A_{j,i}$ denotes the coefficient of $2^j$ in the binary expansion $i-1=\sum_{j=0}^{k-1}A_{j,i}2^{j}$. 
 \item Let $g \in G$ be an element with $g=\prod_{j=1}^{r}(v_{i_j},0)(0,w)$ for $1\leq i_1<i_2<\cdots<i_r\leq k$ then
 $$\chi_{s,i}(g)=\prod_{j=1}^{r}\sign(\chi_{s,i}(v_{i_j},0)).\sign(\chi_{s,i}(0,w)).2^l$$
 \end{enumerate}
\end{theorem}
\begin{proof}
 \begin{enumerate}
  \item We have $\chi_{s,i}=\chi_s \circ f_{s,i}$,  where $\chi_s$ is the irreducible non-linear character of the extraspecial $2$-group $G_s$ and $f_{s,i}$ is the homomorphism as defined in Prop. \ref{number of surjective homomorphisms}.
Let $(v,w)\in G$, then
\begin{align*}
  \chi_{s,i}(v,w) &=  \chi_{s}(f_{s,i}(v,w))\\
                  &=  \chi_{s}(\epsilon_s(v),s(w)-\lambda(v)-h_i(v))
\end{align*}
By Lemma \ref{character_of_extraspecial_2-group}, $\chi_s(x, y)\neq 0$ if and only if $x=0$.
Thus $\chi_{s,i}(v,w)$ = 0 if and only if $\epsilon_s(v)=0$. This happens precisely
when $v\in \rad(b_{s_*(q)})$. Hence the result follows.
\item Let $(0,w)\in Z(G)$ then
\begin{align*}
  \chi_{s,i}(0,w) &=  \chi_{s}(f_{s,i}(0,w))\\
                  &=  \chi_{s}(\epsilon_s(0),s(w)-\lambda(0)-h_i(0))\\
                  &=  \chi_s(0,s(w))
\end{align*}
Let $l$ be defined by $|G_s|=2^{2l+1}$. From Lemma \ref{character_of_extraspecial_2-group} and the
above calculation, we have $$ \chi_{s,i}(0,w)=\left\{ \begin{array}{ll}
                    2^l  & \text{if $s(w)=0$}\\
                           -2^l &  \text{if $s(w)=1$}
                   \end{array} \right. $$
\item Let $\{v_1, v_2,\cdots , v_k\}$ be a basis of $\rad(b_{s_*(q)})$.
Let $A_{j, i}$ be defined by the binary expansion $$i-1=A_{0,i}2^{0}+A_{1,i}2^{1}+\cdots +A_{k-1,i}2^{k-1},$$
where $1\leq i\leq 2^k$. 
Consider a map $\lambda : V \to \F_2$ as in equation \ref{lambda} in section \ref{Subsection_Representations-of-real-special-$2$-groups}. 
We define a map $\theta: \rad(b_{s_*(q)}) \to \F_2$ by 
$$\theta\left(\sum_{j=1}^{r}v_{i_j}\right)=\sum_{j=1}^{r}\lambda(v_{i_j}) \text{ for } 1 \leq i_1<i_2<\cdots<i_r\leq k$$
Notice that the map $\theta$ is nothing but the linear extension to  $\rad(b_{s_*(q)})$ of the restriction of $ \lambda$ to the basis $\{v_1, v_2,\cdots , v_k\}$. Thus $\theta \in \Hom_{\F_2}(\rad(b_{s_*}(q)), \F_2)$.
We recall from the proof of Prop. \ref{number of surjective homomorphisms}
the description of maps $h_i : V \to \F_2$; $1\leq i\leq 2^{k}$.
By definition, $h_i(v_j)=t_i(v_j)$
as $v_j\in \rad(b_{s_*(q)})$. Let $t_i^{\prime}\in \Hom_{\F_2}(\rad(b_{s_*}(q)), \F_2)$ be defined by $t_i^{\prime}(v_j)=A_{j,i}$. Since both $\{t_i: 1\leq i\leq 2^k\}$ and $\{ \theta-t_i^{\prime}: 1\leq i\leq 2^k\}$ denote the same set, namely the set of all the linear maps from $\rad(b_{s_*}(q))$ to $\F_2$, by a suitable  permutation of indices we may assume that $t_i=\theta-t_i^{\prime}$. Thus we have
\begin{align*}
  \chi_{s,i}(v_j,0) &=  \chi_s(f_{s,i}(v_j,0))\\
                  &=  \chi_s(\epsilon_s(v_j),s(0)-\lambda(v_j)-h_i(v_j))\\
                  &=  \chi_s(0,-\lambda(v_j)-t_i(v_j))\\
                  &=  \chi_s(0,-\lambda(v_j)-\theta(v_j)+t_i^{\prime}(v_j))\\
                  &= \chi_s(0,-\lambda(v_j)-\lambda(v_j)+t_i^{\prime}(v_j))\\
                  &= \chi_s(0,t_i^{\prime}(v_j))\\
                  &= \chi_s(0,A_{j,i})
\end{align*}
Therefore
\begin{align*}
    \chi_{s,i}(v_j,0) &= \left\{ \begin{array}{ll}
                      \chi_s(0,1)  & \text{if $A_{j,i}=1$}\\
                       \chi_s(0,0) &  \text{if $A_{j,i}=0$}
                   \end{array} \right.
 = \left\{ \begin{array}{ll}
                      -2^l  & \text{if $A_{j,i}=1$}\\
                       2^l &  \text{if $A_{j,i}=0$}
                   \end{array} \right.
\end{align*}
\item Let $\{v_1, v_2, \cdots, v_k\}$  be a basis of $\rad(b_{s_*(q)})$.
Take $g=\prod_{j=1}^{r}(v_{i_j},0)(0,w) \in G$; where $1\leq i_1<i_2<\cdots<i_r\leq k$  then
\begin{align*}
\varphi_{s,i}(g)
 &=\varphi_{s,i}\left(\prod_{j=1}^{r}(v_{i_j},0)(0,w)\right)\\
 &=\prod_{j=1}^{r}\varphi_{s,i}(v_{i_j},0)\varphi_{s,i}(0,w)
 \end{align*}
 Now from Prop. \ref{character-table}(2) and Prop. \ref{character-table}(3) the value of $\chi_{s,i}(v_j, 0)$ and $\chi_{s,i}(0,w)$ is either $2^{l}$ or $-2^{l}$. From Remark \ref{representations}, we have that $\varphi_{s,i}(v_j,0)$ and $\varphi_{s,i}(0,w)$ are either $\diag(1,1,\cdots,1)$ or $\diag(-1,-1,\cdots, -1)$. Now from above calculation and definition \ref{sign}, we have
 $$\chi_{s,i}\left(\prod_{j=1}^{r}(v_{i_j},0)(0,w)\right)=\prod_{j=1}^{r}\sign(\chi_{s,i}(v_{i_j},0)).\sign(\chi_{s,i}(0,w)).2^l.$$
Hence the result follows.
\end{enumerate}
\end{proof}
We summarize Prop. \ref{character-table} in the following table. Notations of table are same as Prop. \ref{character-table}. Let $|G|=2^n$, $|Z(G)|=2^m$ and $|G_s|=2^{2l+1}$. We recall from the proof of Theorem \ref{complete_list_representations} 
that in this case $|\rad(b_{s_*(q)}|=2^{n-m-2l}$. We fix an ordered basis 
$\{v_1, v_2,\cdots , v_{k}\}$ of  $|\rad(b_{s_*(q)}|$, where $k = {n-m-2l}$.

\begin{center}
{\bf Table A}\\
 \vspace*{0.5cm}

 \resizebox{\linewidth}{!}{%
 \begin{tabular}{|c|c|c|}
\hline 
{\bf Type of element}  & ${\mathbf \chi_{s,i}(v,w)}$ & {\bf Number of elements} \\
\hline 
$\{(v,w) :  v \notin \rad(b_{s_*(q)}\}$ & $0$ & $2^n-2^{n-2l}$ \\
\hline
$\{(0,w) : s(w)=0\}$ & $2^l$& \\
$\{(0,w) : s(w)=1\}$ & $-2^l$& $2^m$\\
\hline

$\{(v_j,0) : A_{j,i}=1\}$ & $-2^l$&  \\
$\{(v_j,0) : A_{j,i}=0\}$ & $2^l$ & $(n-m-2l).2^m$\\
$\{(v_j,w) : 0\neq (0, w) \in Z(G)\}$ & $\sign(\chi_{s,i}(v_j,0)).\sign(\chi_{s,i}(0,w)).2^l$&\\
\hline
 $\displaystyle \prod_{j=1}^{r}(v_{i_j},0)(0,w)$ where & &\\
 $1 \leq i_1<i_2<\cdots<i_r\leq k$ and $r\geq 2$ &  
 $\displaystyle \prod_{j=1}^{r}\sign(\chi_{s,i}(v_{i_j},0)).\sign(\chi_{s,i}(0,w)).2^l$& $(2^{n-m-2l}-(n-m-2l)-1).2^m$\\ 
\hline
\end{tabular}}

\end{center}

\subsection{Conjugacy classes of real special $2$-groups}\label{Subsection_Conjugacy-classes-of-real-special-$2$-groups}
In this section we form conjugacy classes of a real special $2$-group $G$. We frequently use the following
well-known result to distinguish two conjugacy classes.

\begin{proposition}[\cite{JL}, Prop. 15.5]\label{conjugacy-via-characters}
Let $G$ be a finite group and $\chi_1, \chi_2,\cdots, \chi_k$ be the complete set of inequivalent irreducible characters of $G$. 
Then two elements $g, h \in G$ are conjugates if and only if $\chi_i(g)=\chi_i(h)$ for all $1\leq i\leq k$.
\end{proposition}

\rem \label{conjucagy_via_linear_reps}
Let $G$ be a real special $2$-group and $q : V \rightarrow W$ be the quadratic map associated to $G$.
The elements of $W = Z(G)$ form conjugacy classes containing only one element. 
Let $v_1 \neq v_2 \in V$ and $v_1\neq 0,~v_2\neq 0$ in $V$. Then elements of the set $\{(v_1,w) : w\in W\}$ are not conjugate to any element of set $\{(v_2,w) : w\in W\}$. This follows from Prop. \ref{conjugacy-via-characters} using the fact that linear characters of $G$ are the lifts of irreducible characters of $\frac{G}{Z(G)}$. We therefore conclude that the sets $\{(v,w) : w\in W\}$ indexed by non-zero $v\in V$ are mutually disjoint.
This way, we divide the non-central elements of $G$ into $|\frac{G}{Z(G)}|-1$ number of sets
each containing $|Z(G)|$ elements.\\
\begin{theorem}\label{conjugacy_classes_not-in-rad}
Let $G$ be a real special $2$-group and $q: V:=\frac{G}{Z(G)}\to Z(G)=:W$ be the quadratic map associated to $G$.
Let $v \in V$. If $v \notin \rad(b_{s_*(q)})$ for all non-zero linear maps $s:W\rightarrow \F_2$ 
then $\{(v,w) : w\in W\}$ is a conjugacy class of $G$.
\end{theorem}
\begin{proof}
Let $v \notin \rad(b_{s_*(q)})$ for all non-zero linear maps $s:W\rightarrow \F_2$. In view of 
Prop. \ref{conjugacy-via-characters} and Remark \ref{conjucagy_via_linear_reps} it is enough to show that $\chi(v,w) = \chi(v,w')$ for all $w, w' \in Z(G)$ and for 
all irreducible characters $\chi$.

First suppose that $\chi$ is linear. Since linear representations of $G$ are precisely the lifts of irreducible representations of $\frac{G}{Z(G)}$, 
the image of all the elements in the set $\{(v,w) : w\in W\}$ are same under all the linear representations of $G$. \\

Now suppose that $\chi$ is non-linear. Since $v \notin \rad(b_{s_*(q)})$, by Theorem \ref{character-table}(1)
$\chi(v,w)=0$ for all the elements of set $\{(v,w) : w\in W\}$. The result now follows from 
Prop. \ref{conjugacy-via-characters}.
\end{proof}
 
\begin{theorem}\label{conjugacy_classes_rad}
Let $G$ be a real special $2$-group and $q: V:=\frac{G}{Z(G)}\to Z(G)=:W$ be the quadratic map associated to $G$.
Let $v\in V$ and $\mathcal S_v := \{s \in \Hom_{\F_2}(Z(G), \F_2) : v\in \rad(b_{s_*(q)})\}$.
Then the conjugacy class of element $(v,w)\in G$ is $\{(v,w+w^{\prime}) : s(w^{\prime})=0 \text{ for all }s\in \mathcal S_v\}. $ 
\end{theorem}
\begin{proof}
By Remark \ref{conjucagy_via_linear_reps}, we know that the conjugacy class of $(v,w)$ is a subset of 
$\{(v,w^{\prime}) : w^{\prime} \in Z(G)\}$.
Let $w_1\in Z(G)$ be such that $s(w_1)=1$ for some non-zero $s \in \mathcal S_v$. 
Then $$\chi_{s,i}(v,w)=\chi_s(\epsilon_s(v),s(w)-\lambda(v)-h_{i}(v)) =(0,s(w)-\lambda(v)-h_{i}(v))$$ 
$$\chi_{s,i}(v,w+w_1)=\chi_s(\epsilon_s(v),s(w+w_1)-\lambda(v)-h_{i}(v))=(0,s(w)+1-\lambda(v)-h_{i}(v))$$
It is clear that one of $\chi_{s,i}(v,w)$ and $ \chi_{s,i}(v,w+w_1)$ is $\chi_s(0,0)=2^l$, while the other one is $\chi_s(0,1)=-2^l$, where $2^{l}$ is the degree of the character $\chi_{s,i}$.
Thus $\chi_{s,i}(v,w_1) \neq \chi_{s,i}(v,w_1+w)$ and it follows from Prop. \ref{conjugacy-via-characters} that $(v,w)$ and $(v,w+w_1)$ are not conjugates.\\

Let $v\in \rad(b_{s_*(q)})$ and $w^{\prime}\in W$ be such that $s(w^{\prime})=0$ for all non-zero linear maps in $\mathcal S_v$.
Since linear characters of group $G$ are lifts of irreducible characters of $\frac{G}{Z(G)}$, we have $\chi((v,w))=\chi((v,w+w^{\prime}))$ for all linear characters
$\chi$.\\

Let $\chi_{s,i}$ be a non-linear character and $v\notin \rad(b_{s_*(q)})$, then $\chi_{s,i}((v,w))=\chi_{s,i}((v,w+w^{\prime}))=0$.

Finally we consider the non-linear characters $\chi_{s,i}$ such that $v\in \rad(b_{s_*(q)})$. Then as earlier 
$$\chi_{s,i}(v,w)=\chi_s(0,s(w)-\lambda(v)-h_{i}(v))$$
$$\chi_{s,i}(v,w+w^{\prime})=\chi_s(0,s(w)+s(w^{\prime})-\lambda(v)-h_{i}(v))=\chi_s(0,s(w)-\lambda(v)-h_{i}(v))$$
Thus again in this case $\chi_{s,i}((v,w))=\chi_{s,i}((v,w+w^{\prime}))$. By Prop. \ref{conjugacy-via-characters}, $(v,w)$ is conjugate to $(v,w+w^{\prime})$.
\end{proof}

We summarize the types of conjugacy classes of special $2$-group $G$ in the following table. The notations in the table are same as that of Theorem \ref{conjugacy_classes_rad}.

\begin{center}
{\bf Table B}\\
\vspace*{0.5cm}
 \begin{tabular}{|c|c|}
  \hline
  {\bf Type of element} & {\bf Conjugacy class} \\
  \hline
  \begin{tabular}{c}
  $v \notin \rad(b_{s_*(q)})$ for all\\
  $0\neq s\in \Hom(Z(G),\F_2)$ \end{tabular} & $\{(v,w) : w\in Z(G)\}$\\
  \hline
  \begin{tabular}{c}
  $v \in \rad(b_{s_*(q)})$; $s\in\mathcal S_v $ \end{tabular}& $\{(v,w+w^{\prime}) : s(w^{\prime})=0 ~\forall~ s\in\mathcal S_v \}$\\
  \hline
\end{tabular}
\end{center}

\section{Example}\label{example}
In this section we demonstrate through an example that the results proved in earlier sections can
be used to construct the character table of a real special $2$-groups. The example that we consider is that of the
group $G$ defined by
$$G=\langle a,b,c,d,f : a^2=b^2=(ab)^2=d,c^2=(ac)^2=f,d^2=f^2=(bc)^2=(df)^2=1\rangle.$$

We make following observations about $G$.

\begin{itemize}
\item The center of $G$ is $Z(G) := \langle d, f : d^2=f^2=(df)^2=1\rangle$,
and the quotient by the center is $\frac{G}{Z(G)} : = \langle 
\bar{a},\bar{b},\bar{c} : \bar{a}^2=\bar{b}^2=\bar{c}^2=\bar{(ab)}^2=\bar{(ac)}^2=\bar{(bc)}^2=\bar{1}\rangle$.
Both $Z(G)$ and $\frac{G}{Z(G)}$ are elementary abelian $2$-groups.
\item The group $G$ is a special $2$-group as $|G|=32$ and $Z(G)=\Phi(G)=G^{\prime}=\langle d, f : d^2=f^2=(df)^2=1\rangle$.
\end{itemize}

We identify $\frac{G}{Z(G)}$ with a $3$-dimensional vector space $V$ and
 $Z(G)$ with a $2$-dimensional vector space
$W$ over the field $\F_2$. Therefore, as a set, the group $G$ gets identified with 
$V\times W$. Let $\{e_1=(1,0,0), e_2=(0,1,0), e_3=(0,0,1)\}$ be a basis of $V$ and
$\{f_1=(1,0),f_2=(0,1)\}$ be a basis of $W$ over $\F_2$. 
The quadratic map $q:V\to W$ associated to the special $2$-group $G$ is given by
$$q(x,y,z)=(x^2+xy+y^2, z^2+xz); \quad(x,y,z)=x(1,0,0)+y(0,1,0)+z(0,0,1)\in V.$$

We claim that the group $G$ is real. We use Theorem \ref{real-group-criteria} to justify this claim.
Let $v\in V$. We find $0 \neq v^{\prime} \in V$ such that $q(v^{\prime})=q(v+v^{\prime})$ to show that $G$ is indeed
real. The following table explicitly exhibits such $v^{\prime} \in V$ for a given $v \in V$.

\begin{center}
 \begin{tabular}{|c|c|c|}
  \hline
  $v$ & $v^{\prime}$ & $q(v^{\prime})=q(v+v^{\prime})$\\
  \hline
  \begin{tabular}{c}
  $(0,0,0,),(0,1,0)$\\
  $(0,0,1),(0,1,1)$\\
 \end{tabular} & $(1,0,0)$ & $(1,0)$\\
 \hline
\begin{tabular}{c}
  $(1,0,0,),(1,1,0)$\\
  $(1,0,1),(1,1,1)$\\
 \end{tabular} & $(0,1,0)$ & $(1,0)$\\
\hline
\end{tabular}
\end{center} 

It follows therefore that $G$ is real. \\

In the following table we compute the radical $\rad(b_{s_{i*}(q)})$
for each non-zero linear map $s_i:W\to \F_2$.

\begin{center}
 \resizebox{\linewidth}{!}{%
 \begin{tabular}{|c|c|c|c|c|}
 \hline
  Linear map ($s$) & $s \circ q$ & $b_{s_{*}(q)}$ & $ \rad(b_{s_*(q)})$ & $ |\rad(b_{s_*(q)})|$\\
  \hline
  $s_1(w_1,w_2)=w_1$ & 
                         $q(x,y,z)=
                         x^2+xy+y^2$  &  \begin{tabular}{c}
                         $b_{s_{1*}(q)}((x,y,z),(x^{\prime},y^{\prime},z^{\prime}))$\\
                         $=xy^{\prime}+x^{\prime}y$ \\ \end{tabular}  & $\langle e_3 \rangle$ & $2$ \\
  \hline
  $s_2(w_1,w_2)=w_2$ & 
                                       $q(x,y,z)=
                         z^2+xz$ &  \begin{tabular}{c}
                         $b_{s_{2*}(q)}((x,y,z),(x^{\prime},y^{\prime},z^{\prime}))$\\
                         $=xz^{\prime}+x^{\prime}z$\\  \end{tabular}  & $\langle e_2 \rangle$ & $2$ \\      
                         \hline
 $s_3(w_1,w_2)=w_1+w_2$ & \begin{tabular}{c}
                                       $q(x,y,z)=$\\
                         $x^2+x(y+z)+(y+z)^2$\\ \end{tabular} &  \begin{tabular}{c}
                         $b_{s_{3*}(q)}((x,y,z),(x^{\prime},y^{\prime},z^{\prime}))$\\
                         $=x(y^{\prime}+z^{\prime})+x^{\prime}(y+z)$ \\ \end{tabular}  & $\langle e_2+e_3 \rangle$ & $2$ \\   
                         \hline
 \end{tabular}}

\end{center}
\vspace*{0.2cm}
We compute the conjugacy classes of $G$ using the results of 
\S \ref{Subsection_Conjugacy-classes-of-real-special-$2$-groups}.

\begin{center}
 \resizebox{\linewidth}{!}{%
 \begin{tabular}{|c|c|}
  \hline
  $\mathcal{C}_1=\{(0,0)\}$ & $\mathcal{C}_2=\{(0,f_1)\}$\\
  \hline
  $\mathcal{C}_3=\{(0,f_2)\}$ & $\mathcal{C}_4=\{(0,f_1+f_2)\}$\\
  \hline
  $\mathcal{C}_5=\{(e_1,0),(e_1,f_1),(e_1,f_2),(e_1,f_1+f_2)\}$ & $\mathcal{C}_6=\{(e_2,0),(e_2,f_1)\}$\\
  \hline
  $\mathcal{C}_7=\{(e_2,f_2),(e_2,f_1+f_2)\}$ & $\mathcal{C}_8=\{(e_3,0),(e_3,f_2)\}$\\
  \hline
  $\mathcal{C}_9=\{(e_3,f_1),(e_3,f_1+f_2)\}$ & \begin{tabular}{c}
                                   $\mathcal{C}_{10}=\{(e_1,0)(e_2,0),(e_1,0)(e_2,0)(0,f_1),$\\
                                   $(e_1,0)(e_2,0)(0,f_2),(e_1,0)(e_2,0)(0,f_1+f_2)\}$\\
                                  \end{tabular}\\
  \hline
  \begin{tabular}{c}
  $\mathcal{C}_{11}=\{(e_1,0)(e_3,0),(e_1,0)(e_3,0)(0,f_1),$\\ $(e_1,0)(e_3,0)(0,f_2),(e_1,0)(e_3,0)(0,f_1+f_2)\}$\\
  \end{tabular}  & $\mathcal{C}_{12}=\{(e_2,0)(e_3,0),(e_2,0)(e_3,0)(0,f_1+f_2)\}$\\
  \hline
  $\mathcal{C}_{13}=\{(e_2,0)(e_3,0)(0,f_1),(e_2,0)(e_3,0)(0,f_2)\}$ & \begin{tabular}{c}
                                          $\mathcal{C}_{14}=\{(e_1,0)(e_2,0)(e_3,0),(e_1,0)(e_2,0)(e_3,0)(0,f_1),$\\
                                          
                                          $(e_1,0)(e_2,0)(e_3,0)(0,f_2),(e_1,0)(e_2,0)(e_3,0)(0,f_1+f_2)\}$\\ \end{tabular} \\
                                          \hline
                                          
 \end{tabular}}

\end{center}

Now, for each non-zero linear map $s:W\to \F_2$ we compute the regular 
quadratic forms $q_{s}$ up to isometry and determine the extraspecial $2$-groups $G_{s}$
associated to these quadratic forms using Remark \ref{classification-of-extraspecial-2-groups}. 

\begin{center}
\resizebox{\linewidth}{!}{%
\begin{tabular}{|c|c|c|c|c|c|c|}
\hline
Linear map ($s$) & $s\circ q$ & $q_{s}$ & $G_{s}$ & $|G_{s}|$ & Characters & Degree \\
\hline
  $s(w_1,w_2)=w_1$  &  $q(x,y,z)=x^2+xy+y^2$ & $[1,1]$ & $Q_2\circ D_4$ & $8$ & $\chi_{s_1,1},\chi_{s_1,2}$ & $2$\\
  \hline
  $s(w_1,w_2)=w_2$  &  $q(x,y,z)= z^2+xz$ & $[0,0]$ & $D_4\circ D_4$ & $8$ & $\chi_{s_2,1},\chi_{s_2,2}$ & $2$\\
  \hline
   $s(w_1,w_2)=w_1 + w_2$  & 
  \begin{tabular}{c}
                                       $q(x,y,z)=$\\
                         $x^2+x(y+z)+(y+z)^2$\\ \end{tabular} & $[1,1]$ & $Q_2\circ D_4$ & $8$ & $\chi_{s_3,1},\chi_{s_3,2}$ & $2$\\
  \hline
\end{tabular}}
\end{center}
\vspace*{0.5cm}
For each non-zero linear map $s_i : W \to \F_2$ we compute 
$|\rad(b_{s_{i*}(q)})|$ number of non-linear irreducible characters $\chi_{s, j}$ using Prop.\ref{character-table}.
The linear characters of group $G$ are determined using the Remark \ref{linear-representations-of-special-2-group}. \\

This summarizes to the following character table of $G$.\\

\begin{center}
\begin{tabular}{|c|c|c|c|c|c|c|c|c|c|c|c|c|c|c|c|}
\hline
  & $\mathcal{C}_1$ & $\mathcal{C}_2$ &  $\mathcal{C}_3$ & $\mathcal{C}_4$ &  $\mathcal{C}_5$ & $\mathcal{C}_6$ &  $\mathcal{C}_7$ & $\mathcal{C}_8$ &  $\mathcal{C}_9$ & $\mathcal{C}_{10}$ &  $\mathcal{C}_{11}$ & $\mathcal{C}_{12}$ &  $\mathcal{C}_{13}$ & $\mathcal{C}_{14}$ \\
 \hline
 $\chi_1$ & $1$ & $1$ & $1$ & $1$ &$1$ & $1$ &$1$ & $1$ &$1$ & $1$ &$1$ & $1$ &$1$ & $1$ \\
 \hline
 $\chi_2$ & $1$ & $1$ & $1$ & $1$ &$1$ & $1$ &$1$ & $-1$ &$-1$ & $1$ &$-1$ & $-1$ &$-1$ & $-1$ \\
 \hline
 $\chi_3$ & $1$ & $1$ & $1$ & $1$ &$1$ & $-1$ &$-1$ & $1$ &$1$ & $-1$ &$1$ & $-1$ &$-1$ & $-1$ \\
 \hline
 $\chi_4$ & $1$ & $1$ & $1$ & $1$ &$1$ & $-1$ &$-1$ & $-1$ &$-1$ & $-1$ &$-1$ & $1$ &$1$ & $1$ \\
 \hline
 $\chi_5$ & $1$ & $1$ & $1$ & $1$ &$-1$ & $1$ &$1$ & $1$ &$1$ & $-1$ &$-1$ & $1$ &$1$ & $-1$ \\
 \hline
 $\chi_6$ & $1$ & $1$ & $1$ & $1$ &$-1$ & $1$ &$1$ & $-1$ &$-1$ & $-1$ &$1$ & $-1$ &$-1$ & $1$ \\
 \hline
 $\chi_7$ & $1$ & $1$ & $1$ & $1$ &$-1$ & $-1$ &$-1$ & $1$ &$1$ & $1$ &$-1$ & $-1$ &$-1$ & $1$ \\
 \hline
 $\chi_8$ & $1$ & $1$ & $1$ & $1$ &$-1$ & $-1$ &$-1$ & $-1$ &$-1$ & $1$ &$1$ & $1$ &$1$ & $-1$ \\
 \hline
 $\chi_{s_1,1}$ & $2$ & $-2$ & $2$ & $-2$ &$0$ & $0$ & $0$ & $2$ &$-2$ & $0$ &$0$ & $0$ &$0$ & $0$ \\
 \hline
 $\chi_{s_1,2}$ & $2$ & $-2$ & $2$ & $-2$ &$0$ & $0$ & $0$ & $-2$ &$2$ & $0$ &$0$ & $0$ &$0$ & $0$ \\
 \hline
 $\chi_{s_2,1}$ & $2$ & $2$ & $-2$ & $-2$ &$0$ & $2$ & $-2$ & $0$ &$0$ & $0$ &$0$ & $0$ &$0$ & $0$ \\
 \hline
 $\chi_{s_2,2}$ & $2$ & $2$ & $-2$ & $-2$ &$0$ & $-2$ & $2$ & $0$ &$0$ & $0$ &$0$ & $0$ &$0$ & $0$ \\
 \hline
 $\chi_{s_3,1}$ & $2$ & $-2$ & $-2$ & $2$ &$0$ & $0$ & $0$ & $0$ &$0$ & $0$ &$0$ & $2$ &$-2$ & $0$ \\
 \hline
 $\chi_{s_3,2}$ & $2$ & $-2$ & $-2$ & $2$ &$0$ & $0$ & $0$ & $0$ &$0$ & $0$ &$0$ & $-2$ &$2$ & $0$ \\
 \hline
\end{tabular}
\end{center}

\bibliographystyle{plain}

\label{'ubl'}  
\end{document}